\documentclass[reqno]{amsart}

\usepackage{amssymb,xypic,graphicx,hyperref,url}
\usepackage{thmtools,thm-restate,cleveref, mathtools} %%%

\newcommand{\excise}[1]{}

\newtheorem{theorem}{Theorem}
\numberwithin{theorem}{section}
\newtheorem{proposition}[theorem]{Proposition}

\newtheorem{corollary}[theorem]{Corollary}
\newtheorem{lemma}[theorem]{Lemma}

\newtheorem{question}[theorem]{Question}
\theoremstyle{definition}
\newtheorem{definition}[theorem]{Definition}
\newtheorem{remark}[theorem]{Remark}
\newtheorem{example}[theorem]{Example}

\DeclareMathOperator{\LL}{L}
\DeclareMathOperator{\MM}{M}
\DeclareMathOperator{\NN}{{N}}

\DeclareMathOperator{\DD}{{D}}
\DeclareMathOperator{\I}{{I}}
\DeclareMathOperator{\A}{{A}}

\newcommand{\cc}{{\bf c}}

\newcommand{\proofsketch}{\vspace*{-1ex} \noindent {\emph{Proof Sketch.} }}

\usepackage{kbordermatrix}
\kbalignrighttrue
\renewcommand{\kbldelim}{(}
\renewcommand{\kbrdelim}{)}
\renewcommand{\kbrowstyle}{\displaystyle}
\renewcommand{\kbcolstyle}{\displaystyle}

\subjclass[2010]{Primary }
\begin{document}

\title{Chip-firing on general invertible matrices}
\author[J.~Guzm\'an]{Johnny Guzm\'an}
 \email{johnny\_guzman@brown.edu}
 \address{Division of Applied Mathematics, Brown University, Providence, RI 02906}
\author{Caroline Klivans}
\email{klivans@brown.edu}
\address{Division of Applied Mathematics and Department of Computer Science, Brown University, Providence, RI 02906}

\date{\small \today}

 \thanks{}
 
 
 
%%%%%%%%%%%%%%%%%%%%%%%%%%%%%%%%%%%%%
%%%%%%%%%%%%%%%%%%%%%%%%%%%%%%%%%%%%%

\begin{abstract}
We propose a generalization of the graphical chip-firing model allowing for the redistribution dynamics to be governed by any invertible integer matrix while maintaining the long term critical, superstable, and energy minimizing behavior of the classical model.

%% \vspace{.5in}
%% \noindent Corresponding author: Caroline Klivans\\
%% Telephone: 401-863-1460 / Fax: 401-863-1355\\
%% Email: klivans@brown.edu\\

\end{abstract}

\maketitle

%%%%%%%%%%%%%%%%%%%%%%%%%%%%%%%%%
%%%%%%%%%%%%%%%%%%%%%%%%%%%%%%%%%
%%%%%%%%%%%%%%%%%%%%%%%%%%%%%%%%%
\section{Introduction}

The classical chip-firing model describes the dynamics of a Laplacian
action on a graph and can be seen as a discrete diffusion process, see e.g.~\cite{BTW, Dhar, gabe}.  For
a graph with $n$ vertices, configurations are integer vectors
${\bf{c}} \in \mathbb{Z}^n$ where $c_i$ is interpreted as the number
of chips at vertex $i$.  At each time step, chips are redistributed
via local rules.  For a configuration ${\bf c}$, if a vertex $i$ has
at least as many chips as it has neighbors, $i$ {fires} sending one
chip to each of its neighbors, forming the new configuration ${\bf
  c}'$.  This action is easily identified as subtracting the
appropriate row of the graph Laplacian: ${\bf c}' = {\bf c} -
\LL^Te_i$, where $\LL$ is the graph Laplacian and $e_i$ is the $i$th
indicator vector.

This model is well studied and has a strong theory of long term
behavior in terms of self-organized criticality~\cite{BTW, Dhar}, superstability~\cite{Perlman, Primer} and
energy minimization~\cite{BS, GK}.  In particular, every initial configuration
eventually stabilizes to a unique stable configuration - the
order of firings does not matter.

Already in~\cite{gabe}, Gabrielov broadened the setting
for chip-firing realizing that one could consider a finite set of
states and the dynamics governed by any \emph{avalanche finite}
matrix.  The system did not need to come from an underlying graph; the
firing rule is governed by the matrix.  Such systems similarly
converge to a set of stable configurations regardless of the
order of firings.

In more recent work, the present authors~\cite{GK} recognized
avalanche-finite matrices as precisely \emph{$\MM$-matrices}, a well
known class of matrices from other contexts and continued to expand
the theory of chip-firing in this setting.  In particular, the theory
of criticality, superstability and energy minimization extends to all $\MM$-matrices. 
The name \emph{avalanche finite}, however, is quite appropriate.  If
one attempts to chip-fire using a matrix outside of this class,
Gabrielov showed that not all initial configurations will eventually
stabilize.  Some state will always be able to fire.

In this paper we propose a significant broadening of the chip-firing
model allowing for the redistribution dynamics to be governed by
\emph{any invertible integer matrix} while maintaining the long term critical,
superstable, and energy minimizing behavior of the classical model.
To overcome the seeming contradiction with Gabrielov's work, we change
the scope of \emph{valid} configurations over which chip-firing
occurs.  In classical chip-firing, valid configurations are precisely
non-negative integer vectors, i.e. integer points in the closure of
the positive orthant.  In our more general setting, the positive
orthant is replaced by a different cone.  As a consequence, valid (and
hence also critical and superstable) configurations may contain
negative entries.  

The general set-up is as follows.  For a fixed invertible matrix
$\LL$, one must choose an $\MM$-matrix $\MM$.  The chip-firing model
will be dictated by the pair $(\LL,\MM)$.  The matrix $\MM$ can be any
$\MM$-matrix, it need not depend on $\LL$.  Hence for a fixed
distribution matrix $\LL$, one can in fact define many
chip-firing models.  The valid configurations are a collection of
integer points, denoted $S^+$, determined by the matrix $\LL\MM^{-1}$.  Given a configuration ${\bf c}\in S^+$, a site $i$ is allowed
to fire if the resulting configuration ${\bf c}' = {\bf c} - {\LL}e_i$
remains in $S^+$.  The well-behaved dynamics of this model are
established by relating chip-firing via $\LL$ over $S^+$ to
chip-firing via $\MM$ over a related set of (not necessarily integer)
configurations.  As in the graphical case, critical, superstable and energy
minimizing configurations of $S^+$ all form systems of representatives of the
\emph{critical group} of $\LL$, $K(\LL) := \textrm{ coker}(\LL)$.

As a special case, one easily recovers classical chip-firing on a
graph.  Let $\LL$ be (the transpose of) a reduced graph Laplacian. Then
$\LL$ is itself an $\MM$-matrix.  Thus for such an $\LL$ (or for any
initial $\MM$-matrix) one can choose $\MM = \LL$ giving the pairing
$(\LL,\LL)$.
As we will see in Section~\ref{graphs}, $S^+$ then reduces to
$\mathbb{Z}^n_{\geq 0}$ and we recover the well known chip-firing model
for graphs and digraphs.

As another special case, one can always set $\MM = \I$, since the
identity matrix is an $\MM$-matrix.  The pairing $(\LL,\I)$ is
particularly interesting.  The superstable configurations for $(\LL,
\I)$ are precisely the integer points of the fundamental
parallelepiped of the lattice generated by $\LL$, see Section~\ref{lattice}.  It is well-known
that these integer points form a system of representatives for
coker$(\LL)$.  The theory here now gives a dynamical system which
achieves these configurations.

In Section~\ref{sec:general}, we explain the general set-up for
chip-firing over an invertible integer matrix $\LL$ with respect to a
choice of $\MM$-matrix $\MM$.  In Sections~\ref{sec:critical},
\ref{sec:energy}, and \ref{sec:superstable} we show the existence and
uniqueness of critical, energy minimizing, and superstable
configurations respectively.  Finally, in Section~\ref{sec:cases} we consider how
the long term dynamics vary under different choices of $\MM$ including $\MM = \LL$ and $\MM=\I$.
  We also consider the special case of
chip-firing in higher dimensions where $\LL$ is a reduced
combinatorial Laplacian of a simplicial complex as in~\cite{DKM}.

%%%%%%%%%%%%%%%%%%%%%%%%%%%%%%%%
\section{General Chip-Firing}\label{sec:general}

Let $\LL$ be an $n \times n$ invertible matrix with integer
coefficients.  Consider $n$ as the number of states of a system.  A
\emph{configuration} of the system is any integer vector $\cc \in \mathbb{Z}^n$.
The value $\cc_i$ is the amount of commodity (chips, dollars, flow,
etc.) at site $i$.  $\LL$ dictates the redistribution of the commodity
among the sites of the system.

Now let $\MM$ be an $n \times n$ M-matrix with real entries and define
$\NN := \LL\MM^{-1}$. The matrix $\NN$ will dictate which configurations are
\emph{valid} configurations of the system.
In particular, given ${\NN}$ define the set:

\begin{equation*}
S^+ =\{{\NN}x \, | \, {\NN}x \in \mathbb{Z}^n, x \in \mathbb{R}^n_{\geq 0} \}.
\end{equation*}
$S^+$ consists of all integer vectors formed by multiplying $\NN$ by a non-negative real vector.
Note that without the non-negativity condition on $x$, this set would simply be all of
$\mathbb{Z}^n$ since $\NN$ is invertible.
The set of integer points in $S^+$ are the valid configurations for
the pairing $(\LL,\MM)$, under the dynamics of $\LL$.  In classical chip-firing on a graph, $S^+$
is exactly the non-negative integer points.

For a valid configuration $\cc \in S^+$, a site $i$ is ready to fire
if the configuration $\cc' = \cc - {\LL}e_i$,  resulting from subtracting
the $i$th row of $\LL$,  is also valid, i.e. if $\cc'$ remains in $
S^+$.  A configuration is called \emph{stable} if no site is ready to
fire.  Similarly, a set of sites can multifire (with multiplicity) if
the resulting configuration $\cc'' = \cc - {\LL}z$ is also valid,
i.e. if $\cc''$ remains in $S^+$, where $z_i$ is the number of times
site $i$ fires.  A configuration in which no set of sites can
multifire is called \emph{superstable}.  We will investigate the
properties of  superstable configurations in the following
sections along with notions of critical and energy minimizing
configurations.  In particular, we will consider these configurations 
with respect to the following equivalence relation on configurations.  

\begin{definition}For two configurations $f,g \in \mathbb{Z}^n$, define $f \sim_{\LL} g$ if $f=g-{\LL}z$ for some $z \in \mathbb{Z}^n$. 
\end{definition}
 Namely, the equivalence classes form the elements of the cokernel of $\LL$, coker($\LL$) = $\mathbb{Z}^n / \textrm{im} \LL$, the critical group of ${\LL}$.  
Each of the collections: superstable, critical, and energy minimizers will form a system of representatives under this equivalence relation.  Namely, there exist unique critical, energy minimizing and superstable configurations per equivalence class in $S^+$, see Theorems~\ref{thm:critical}, \ref{thm:unique}, and \ref{thm:SS}. 

These results will follow naturally from a second chip-firing system
running in parallel with distribution matrix $\MM$. Importantly, $\MM$ is an $\MM$-matrix, which are widely used in many contexts such as economics and scientific computing, see e.g.~\cite{BurmanErn, CiarletRaviart, Leontief} and~\cite{Plemmens} and references therein.  There are many known characterizations of $\MM$-matrices.  We mention the three most important in our context.  
\begin{definition}\label{def:M}

An $n \times n$ non-singular matrix $\MM$ such that $\MM_{i,j} \leq 0$ for all $i \neq j$ and $\MM_{ii} > 0$ for all $i$, is an $\MM$-matrix if any of the following equivalent conditions hold:
\begin{enumerate}
\item $\MM$ is avalanche finite. 
\item All entries of $\MM^{-1}$ are non-negative.
\item There exists $x \in \mathbb{R}^n$ such that $x \geq 0$ and ${\MM}x$ has all positive entries.  
\end{enumerate}
\end{definition}
The equivalence of condition (a) is due to Gabrielov~\cite{gabe} and
refers to a system with non-negative valid configurations.  We refer to Plemmons~\cite{Plemmens} for conditions (b) and (c) and many others. 

The valid configurations corresponding to a choice of $\MM$ are the vectors that generate points in $S^+$. 
 Specifically,
define
\begin{equation*}
R^+=\{ x \in \mathbb{R}^n \, | \, {\NN}x \in \mathbb{Z}^n, x \geq 0 \}.
\end{equation*}
The elements of $R^+$ are non-negative real vectors such that ${\NN}x$ is integer, equivalently ${\NN}x$ is in $S^+$.  The elements of $R^+$ are the valid configurations for the pair $(\LL,\MM)$ under the dynamics of $\MM$. 
\begin{definition} For two configurations $x,y \in \mathbb{R}^n$, define $x \sim_{\MM} y$ if $x=y-{\MM}z$ for some $z \in \mathbb{Z}^n$. 
\end{definition}
 
\begin{proposition}
  Let $f,g \in S^+$ and $x,y \in R^+$ such that $f = {\NN}x$ and $g = {\NN}y$.  Then
  $$ f \sim_{\LL} g  \, \textrm{ iff } \, x \sim_{\MM} y. $$
    \end{proposition}
\begin{proof}   Suppose that $f \sim_{\LL} g$ then we have the following relations:
  \begin{align*}
    f &= g - {\LL}z\\
    {\NN}x &= {\NN}y - {\LL}z\\
    \NN^{-1}{\NN}x &= \NN^{-1}{\NN}y - \NN^{-1}{\LL}z\\
    x &= y - \MM\LL^{-1}{\LL}z\\
    x &= y - {\MM}z.
    \end{align*}
  \end{proof}

Hence we see that (multi-)firing on configurations in $S^+$ via $\LL$
is equivalent to (mulit-)firing on configurations in $R^+$ via $\MM$.
The chip-firing dynamics in the second case are not the same as
chip-firing on $\MM$-matrices as in~\cite{GK}, where the set of valid
configurations is always $\mathbb{Z}_{\geq 0}^n$.  However, we will
see that much of the nice behavior carries over when considering a
general $R^+$.
We end this section with an example.

\begin{example} \label{ex:running}
  Let
\[
\LL = \begin{pmatrix*}[r]
  2 & -1 & 1 \\
  -1 & 2 & -1 \\
  1 & -1 & 2
\end{pmatrix*}
\textrm{ and }
\MM = \begin{pmatrix*}[r]
  3 & -1 & -1 \\
  -1 & 3 & -1 \\
  -1 & -1 & 3
\end{pmatrix*}.
\textrm{ Then }
%\[
\NN = \begin{pmatrix*}[r]
  1 & .25 & .75 \\
  -.25 & .5 & -.25 \\
  .75 & .25 & 1
\end{pmatrix*}.
\]
The valid configurations, $S^+$, consist of integer points of the form ${\NN}x$
such that $x \geq 0$. Vectors in $S^+$ include
$ \begin{pmatrix*}[r] 0 \\ 0 \\ 0 \end{pmatrix*}$,
$ \begin{pmatrix*}[r] 1 \\ 0 \\ 1 \end{pmatrix*}$ and
$ \begin{pmatrix*}[r] 3 \\ -1 \\ 4 \end{pmatrix*}$.

\noindent On the other hand,
$\begin{pmatrix*}[r] 0 \\ 0 \\1 \end{pmatrix*}$ and   $\begin{pmatrix*}[r] 1 \\ -1 \\1 \end{pmatrix*}$
are not in $S^+$ since \\
$\NN^{-1} \begin{pmatrix*}[r] 0 \\ 0 \\1 \end{pmatrix*}$ =
$\begin{pmatrix*}[r] -1.75 \\ .25 \\ 2.25 \end{pmatrix*} \ngeq 0$ and 
$ \NN^{-1} \begin{pmatrix*}[r] 1 \\ -1 \\1 \end{pmatrix*}$ = $\begin{pmatrix*}[r] .75 \\ -1.25 \\ .75 \end{pmatrix*} \ngeq 0$.
\end{example}

\section{Criticality} \label{sec:critical}

We first investigate the long term behavior of successively firing one state at a time.   Assume throughout that the total number of sites in the system is $n$ and let $e_i$ be the $i$th standard basis vector of $\mathbb{R}^n$.    A site is stable if no individual site can validly fire:  
\begin{definition} A configuration $f \in S^+$ is \emph{stable} if $f-{\LL}e_i \notin S^+$ for $1 \leq i \leq n$. 
\end{definition}
\begin{definition}   A configuration $x
\in R^+ $ is \emph{stable} if $x-{\MM}e_i \notin R^+$ for $1 \leq i \leq n$.
\end{definition}
Note that if $f={\NN}x$ with $f \in S^+$ and $x \in R^+$, then $f$ is
stable if and only if $x$ is stable.

\begin{definition} A configuration $f \in S^+$ is \emph{reachable} if there exists a configuration $g \in S^+$ such that
  \begin{itemize}
  \item[(i)] $g - {\LL}e_i \in S^+$ for $1 \leq i \leq n$ \textrm{ and}
  \item[(ii)] $ f = g - \sum_{j = 1}^k {\LL}e_{i_j}  \textrm{ such that } g - \sum_{j = 1}^l {\LL}e_{i_j} \in S^+ \textrm{ for all } l \leq k.$
    \end{itemize}
\end{definition}
Namely, a configuration $f$ is reachable if (i) there exists a sufficiently large configuration $g$ such that (ii) $f$ can be reached from $g$ by a sequence of valid individual firings.  A configuration in $R^+$ is \emph{reachable} if there exists a configuration $y \in R^+$ that satisfies the analogous conditions (i) and (ii) under the $\MM$ operator.

\begin{definition}
  A configuration in $S^+$ or $R^+$ is \emph{critical} if it is both stable and reachable.
\end{definition}
Note that if $f={\NN}x$ with $f \in S^+$ and $x \in R^+$, then $f$ is
critical if and only if $x$ is critical.

The original work arising from the sandpile literature focused on
critical configurations and the phenomenon of \emph{self-organized
  criticality}.  Starting at any sufficiently large initial
configuration and performing valid individual firings, a stable
configuration is eventually obtained, i.e. there are no infinite
sequences of valid firings (condition (a) from
Definition~\ref{def:M}).  Moreover, from a fixed initial
configuration, the same stable configuration is always reached
regardless of the choice and ordering of the individual firings.
Equivalently, critical configurations exist and are unique per
equivalence class under the equivalence relation defined by the image
of the toppling matrix.  This fundamental behavior was considered by
Bak, Tang and Wiesenfeld~\cite{BTW} for grid graph networks, by
Dhar~\cite{Dhar} and Speer~\cite{Speer} for general graphical networks and
Gabrielov~\cite{gabe} for $\MM$-matrix systems.  The behavior
continues to hold in our framework and follows from checking that the
results of~\cite{Dhar, gabe}  extend to the case of restricted valid
configurations.

\begin{theorem}[See {\cite[Theorems 1.2, 2.9]{gabe}} ] \label{thm:critical} Let $(\LL,\MM)$ be a fixed pairing of an invertible integer matrix $\LL$ and $\MM$-matrix $\MM$.  Then, in $S^+$ (resp. $R^+$), critical configurations exist and are unique per equivalence class under the relation $\sim_{\LL}$ (resp. $\sim_{\MM}$).  
\end{theorem}

\proofsketch
Suppose that $x \in R^+$. Then two conditions hold, ${\NN}x \in \mathbb{Z}^n$ and $x \geq 0$.  Firing a state $i$ yields the new configuration $x' = x - {\MM}e_i$.  Importantly, the new configuration $x'$ is such that ${\NN}x'$ remains integer.    To see this, expand  ${\NN}x'$ as follows: 
\begin{align*}
{\NN}x' &= {\NN}(x - {\MM}e_i)\\
&= {\LL}{\MM}^{-1}x - {\LL}{\MM}^{-1}{\MM}e_i\\
&= {\NN}x - {\LL}e_i. 
\end{align*}
${\NN}x$ is integer by assumption and ${\LL}e_i$ is integer because $\LL$ is an integer matrix.  Hence, starting at a configuration in $R^+$, checking to see if the resulting configuration is valid only requires checking if it remains non-negative.  This is the only condition used on the configurations in the referenced Theorems.  
\qed

Note, this implies that the number of critical configurations is equal to the det$(\LL)$, the size of the coker($\LL$).

\begin{example} \label{ex:critical}
 Returning to the pairing $(\LL,\MM)$ from Example~\ref{ex:running}, the critical configurations (of $S^+$) are:  $$\begin{pmatrix*}[r] 4 \\ -1 \\ 4 \end{pmatrix*}, \begin{pmatrix*}[r] 4 \\ 0 \\ 4 \end{pmatrix*}, \begin{pmatrix*}[r] 5 \\ 0 \\ 5 \end{pmatrix*}, \begin{pmatrix*}[r] 5 \\ -1 \\ 5 \end{pmatrix*}.$$ 
\end{example}

\section{Energy Minimization} \label{sec:energy}

In this section we define an energy form on configurations and
consider the relationship between chip-firing and energy as well as
the class of valid configurations which minimize energy.  These ideas
were first developed by Baker and Shokrieh in \cite{BS} and shown to
generalize to $\MM$-matrices in \cite{GK}.  Again, the difference in this context
is that the valid configurations are the points of 
$R^+$ and that $R^+$ is no longer necessarily equal to
$\mathbb{Z}_{\geq 0}^n$.  The technical results of this section will
allow us to establish the existence and uniqueness of superstable
configurations in the next section.  The presentation follows~\cite{GK}.

Given an $\MM$-matrix $\MM$ and a vector $q \in R^+$, define the following energy,
\begin{equation} E(q) = ||\MM^{-1}q||^2_2.
\end{equation}
Below we prove that $E(q)$ has a unique minimizer per equivalence class, see Theorem~\ref{thm:unique}.
To prove this claim, we relate energy to chip-firing in the following two Lemmas. First we make a definition.   Given $z \in \mathbb{Z}^n$ define $z^+ \in \mathbb{Z}^n_{\ge 0}$ by
$z^+_i = z_i$ if $z_i \ge 0$ and $0$ otherwise.  
\begin{lemma}\label{old3.2}
If $x,y \in R^+$ and $y = x - {\MM}z$ then the configuration $h = x - {\MM}z^+$ is also in $R^+$.
\end{lemma}
\begin{proof}
To show that $h$ is in $R^+$, we need to show that $h \geq 0$ and ${\NN}h \in \mathbb{Z}^n$.  
The non-negativity of $h$ follows verbatim from ~\cite[Lemma 3.2]{GK}.
Integrality follows as in Theorem~\ref{thm:critical}. 
\end{proof}
The next result appears as Lemma 3.3 in~\cite{GK}.  It does not consider the domain of the two configurations in question and holds in our more general set-up. 
\begin{lemma}\cite[Lemma 3.3]{GK} \label{old3.3}
Let $\MM$ be an $\MM$-matrix and $y=x-{\MM}z$, then
\begin{equation*}
E(y)=E(x)+z^t z-2z^t \MM^{-1}x=E(x)-z^tz-2z^t\MM^{-1}y. 
\end{equation*}
\end{lemma}
We can now prove the uniqueness of energy minimizers within the valid
configurations of an equivalence class under the relation defined by an $\MM$-matrix $\MM$.  

\begin{theorem} \label{thm:unique}
Let $x$ be a configuration in $R^+$, then the minimization problem 
\begin{equation}\label{Eq:min} \min_{y \sim_{\MM} x, \, \, y \in R^+}   \|\MM^{-1} y\|_{2}^2 \end{equation} has a unique minimizer.  
\end{theorem}
\begin{proof}
Suppose $w, v \in R^+$ and  $w \sim_{\MM} x$ and $v \sim_{\MM} x$ both minimize
Equation~\ref{Eq:min}.  Then $w \sim_{\MM} v$ and there exists $z$ such
that $w = v - {\MM}z$.  Since $w$ is in $R^+$, Lemma~\ref{old3.2} gives
that $h = v - {\MM}z^+$ is also in $R^+$.  
Now, Lemma~\ref{old3.3} gives:
$$ E(h) = E(v) - (z^+)^t(z^+) - 2(z^+)^t\MM^{-1}h.$$ The last quantity,
$2(z^+)^t\MM^{-1}h$,  is positive.  To see this, note that $h \in R^+$
implies $h \geq 0$, $\MM^{-1}$ is non-negative matrix, and $z^+$ is an all non-negative vector.  Therefore, 
$$ E(h) \leq E(v) - (z^+)^t(z^+).$$
On the other hand, $v$ is an energy minimizer, therefore $z^+ = 0$.  
Now consider the relation between $w$ and $v$, the two supposed distinct minimizers,  
$$E(w) = E(v) + z^tz - 2z^t\MM^{-1}v.$$
If $z \leq 0$, then $z^tz$ is positive, $2z^t\MM^{-1}v$ is negative, and $E(v) < E(w)$ unless $z = 0$.  Thus $z=0$ and $w = v$.     
\end{proof}
\begin{corollary}\label{uniquef}
Let $f$ be a configuration in $S^+$, then the minimization problem
\begin{equation}\label{Eq:minS} \min_{g \sim f, g \in S^+}   \|{\LL}^{-1} g\|_{2}^2 \end{equation} has a unique minimizer.  
\end{corollary}
\begin{proof}
Suppose $h,k \in S^+$ and  $h \sim_{\LL} f$ and $k \sim_{\LL} f$ both minimize
Equation~\ref{Eq:minS}.  Since $h,k \in S^+$, $h = {\NN}x$ and
$k = {\NN}y$ for some $x$ and $y$ in $R^+$.  Furthermore, $h
\sim_{\LL} k$ implies that $x \sim_{\MM} y$.

First note that if $h$ and $k$ are both minimizers, then $x$ and
$y$ are both minimizers:  If $x$ is not a minimizer then there
exists an $m$ which is a minimizer such that $m = x - {\MM}z$.
Multiplying by $\NN$ gives ${\NN}m = {\NN}x - {\NN}{\MM}z \implies g = h - {\LL}z$, where
$g = {\NN}m \in S^+$.  But then $E(g) < E(h)$ contradicting the
minimality of $h$.

 On the other hand, $\|{\LL}^{-1} h \|_2^2 = \|{\LL}^{-1} k \|_2^2 $ and
 therefore $\|{\LL}^{-1} {\NN}x \|_2^2 = \|{\LL}^{-1} {\NN}y \|_2^2 $,
 i.e. $$\|{\MM}^{-1}x \|_2^2 = \|{\MM}^{-1}y \|_2^2.$$ But
 Equation~\ref{Eq:min} has a unique minimizer by
 Theorem~\ref{thm:unique}.
\end{proof}

\section{Superstability} \label{sec:superstable}

Recall that a stable configuration $\cc$ is a valid configuration in
which no (individual) site can fire such that the resulting
configuration is still valid, i.e. $\cc' = {\LL}e_i \notin S^+$ for
all $i$.  Superstability extends this notion to account for subsets of sites
firing simultaneously and with multiplicity.  
\begin{definition} A configuration $f \in S^+$ is \emph{superstable} if $f-{\LL}z \notin S^+$ for every
$z \in \mathbb{Z}^n$, $z \ge 0$ and $z \neq 0$. 
\end{definition}
\begin{definition}   A configuration $x
\in R^+$ is \emph{superstable} if $x-{\MM}z \notin R^+$ for every $z \in
\mathbb{Z}^n$, $z \ge 0$ and $z \neq 0$.
\end{definition}
Note that if $f={\NN}x$ with $f \in S^+$ and $x \in R^+$, then $f$ is
superstable if and only if $x$ is superstable.

\begin{example}\label{graphexample}

  Let $\LL$ be the reduced graph Laplacian of the graph in
  Figure~\ref{fig:SS} where we have deleted the row and column of the Laplacian corresponding to the sink vertex. Let $\MM = \LL$.  This choice for $\MM$ is
  valid because reduced graph Laplacians are themselves $\MM$-matrices, see Section~\ref{graphs}.  Under
  the pairing $(\LL, \LL)$, the valid configurations are all
  non-negative vectors, $S^+ = \mathbb{Z}^3_{\geq 0}$.  The
  configuration shown in Figure~\ref{fig:SS} is stable - no individual
  site can fire with the resulting configuration remaining valid.  The
  configuration is however not superstable - the two sites containing $2$
  chips can simultaneously fire with the resulting configuration remaining valid.  
\end{example}
\begin{figure}[h] 
  \includegraphics[width=1.8in]{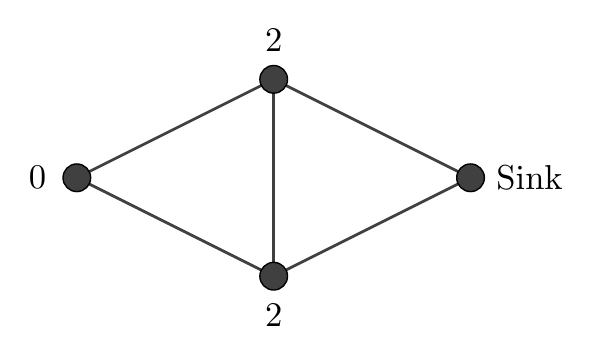}
  \caption{A stable but not superstable configuration.}
 \label{fig:SS}
\end{figure}
\begin{remark}In the graphical case, if the graph is an Eulerian
directed graph, then  it is sufficient to consider subset firing
without multiplicity; i.e., in the definition above, $\mathbb{Z}^n$
may be replaced by $\{0,1\}^n$.  The more general definition above is necessary for
any non-Eulerian graph or more generally for any $\MM$-matrix, see
\cite[Section 4.1]{GK}.
\end{remark}

As with critical and energy minimizing configurations, superstable
configurations exist and are unique per equivalence class of the coker($\LL$).  We
establish this by showing that $x \in R^+$ is superstable if and only
if it is a minimizer of \eqref{Eq:min} and therefore $f \in S^+$ is
superstable if and only if it is a minimizer of \eqref{Eq:minS}.  The following appears essentially as Theorem 4.6 of~\cite{GK}, we reproduce it here for completeness.

\begin{theorem} \label{thm:SS}
Fix an invertible $n \times n$ matrix $\LL$ and an $n \times n$ $\MM$-matrix $\MM$.  A configuration $x \in R^+$ is superstable with respect to $\MM$ if and only if $x$ is an energy-minimizer, i.e., a minimizer of Equation~\ref{Eq:min}.  
\end{theorem}
\begin{proof}
Let $x$ be superstable.  Let $y \in R^+$ be equivalent to $x$.  Then $x = y - {\MM}z$ for some $z$ and $h = y - {\MM}z^+$ is also in $R^+$ by Lemma~\ref{old3.2}.  But $x$ is superstable and hence $z^+$ must be $0$.   Now we can compare the energies of $x$ and $y$:
$$E(y) = E(x) - z^tz - 2z^t\MM^{-1}w.$$
This implies $E(y) < E(x)$ which is a contradiction unless $z = 0$.  

Now let $x$ be an energy-minimizer.  Suppose $x$ is not superstable.  Then there exists $z_0$ such that $x - {\MM}z_0 \in R^+$ with $z_0 > 0$.  Let $y = x - {\MM}z_0 \in R^+$.  Comparing energies gives:
$$E(y) = E(x) - {z_0}^t{z_0} - 2z_0^t{\MM}^{-1}y.$$
This implies $E(y) < E(x)$ which is again a contradiction.  
\end{proof}

\begin{example} \label{ex:SS} Returning to the pairing $(\LL,\MM)$ from Example~\ref{ex:running}, the superstable configurations (of $S^+$) are:  $$\begin{pmatrix*}[r] 1 \\ 0 \\1 \end{pmatrix*}, \begin{pmatrix*}[r] 1 \\ 1 \\1 \end{pmatrix*}, \begin{pmatrix*}[r] 0 \\ 0 \\ 0 \end{pmatrix*}, \begin{pmatrix*}[r] 2 \\ 1 \\ 2 \end{pmatrix*}.$$ 
  \end{example}

  In the classical model, if $\LL$ is the
distribution matrix, define $D^{\LL}$ to be the vector given by
$D_i^{\LL} = {\LL}_{ii} - 1$ for all $i$.  Then a configuration $c$ is
critical if and only if $D^{\LL} - c$ is superstable. 

This duality does not hold in the more general case.  Moreover, there
seems to be no simple linear relation.  Consider, for example, the
critical and superstable configurations from Example~\ref{ex:critical}
and~\ref{ex:SS}.  Note that in $S^+$ there is a component-wise maximum
critical configuration, but subtracting in this configuration does not
give the collection of superstables.
\begin{question} Is there a duality between critical and superstable configurations?
\end{question}

\section{Special Cases} \label{sec:cases}
In this section, we consider three important special cases: (1) $\LL$ is a reduced graph Laplacian (2) $\MM$ is the identity matrix and (3) $\LL$ is a reduced combinatorial Laplacian of a simplicial complex.

\subsection{Graph Laplacians ($\LL,\LL$)}\label{graphs}
In classical chip-firing, for example as in \cite{Biggs, BLS, Dhar},
there is an underlying graph to the system.  States correspond to
vertices of the graph and chips are associated to vertices.  The
distribution dynamics are given by the graph Laplacian $\DD - \A$,
where $\DD$ is the diagonal matrix of degrees of vertices and $\A$ is
the adjacency matrix of the graph.  For a graph with $n$ vertices, the
graph Laplacian is an $n \times n$ {singular} matrix.  To fit our
context, one must introduce a sink vertex -- any vertex can be
declared the sink and then is no longer allowed to fire.  The distribution matrix is
 a \emph{reduced graph Laplacian},  the row and column corresponding to the sink is deleted. 
  The valid configurations consist of all
non-negative integer vectors.

This special case is recovered by first noting that reduced graph
Laplacians are in fact $\MM$-matrices~\cite{gabe, GK} and then pairing
the reduced Laplacian with itself.  Let $\LL$ be a reduced graph
Laplacian (or more generally any $\MM$-matrix). In this case, one can choose $\MM=\LL$ for the pairing
$(\LL,\LL)$.  Under this pairing,
$$ \NN = {\LL}{\LL}^{-1} = \I, \qquad S^+ = \mathbb{Z}^{n-1}_{\geq 0}, \qquad R^+ = \mathbb{Z}^{n-1}_{\geq 0} $$ and we recover the usual graphical chip-firing
dynamics.  The stable and superstable conditions of not leaving $S^+$
reduce to the conditions that any (multi-) firing should yield a
non-negative configuration.
Example~\ref{graphexample} uses the pairing $(\LL, \LL)$ on a small graph.

\subsection{Fundamental parallelepipeds ($\LL, \I$) } \label{lattice}
The identity matrix $\I$ is an $\MM$-matrix.  Therefore, for any fixed
$\LL$, we can always take the pairing $(\LL,\I)$. In fact, any
positive diagonal matrix $\DD$ is an $\MM$-matrix.

\begin{proposition}\label{prop:diagonal}
Suppose $\DD$ is a non-negative diagonal matrix.  If $\MM$ and ${\DD}{\MM}$ are $\MM$-matrices, then the superstable
and critical configurations of $S^+$ are the same under the pairing
$(\LL,\MM)$ and $(\LL, \DD{\MM})$.
  \end{proposition}
\begin{proof}
  We will show that the collection of valid configurations is the same under the two pairings.  Let $S^+$ correspond to $(\LL,\MM)$ and $S^+_{\DD}$ correspond to $(\LL, {\DD}{\MM})$.

  Suppose $f \in S^+$, then $f = {\NN}x = {\LL}{\MM}^{-1}x$ for some $x \geq 0$.  For $f$ to be in $S^+_{\DD}$, we must have $f = {\LL}({\DD}{\MM})^{-1}y$ for some $y \geq 0$.  Solving for $y$, we see that:
\begin{align*}
    {\LL}{\MM^{-1}}{\DD^{-1}}y &= {\LL}{\MM^{-1}}x\\
    {\DD}^{-1}y &= x\\
    y &= {\DD}x. 
    \end{align*}  
Since $x \geq 0$ and ${\DD}$ is positive, ${\DD}x$ is non-negative and $f \in S^+_{\DD}$.

Similarly, suppose $g \in S^+_{\DD}$ then $g = {\LL}{\MM}^{-1}{\DD}^{-1}w$ for some $w \geq 0$.  In order for $g$ to be in $S^+$, the vector $v = {\DD}^{-1}w$ must be non-negative, which again holds because ${\DD}$ is a positive diagonal matrix. 
  \end{proof}

By
Proposition~\ref{prop:diagonal}, the critical and superstable
configurations of $S^+$ are the same for $(\LL, \I)$ and $(\LL, \DD)$.
For the remainder of the section we assume that $\MM = \I$ for simplicity.
Suppose that $\MM = \I$, then 
\begin{equation*}
 \NN = \LL   \textrm{ and } S^+=\{ {\LL}x \, |  \, {\LL}x \in \mathbb{Z}^n, x  \in \mathbb{R}^n_{\geq 0} \}.
\end{equation*}
Given an invertible matrix $\LL$, define 
the \emph{lattice} generated by the columns of
$\LL$: $$\mathcal{L}(\LL) = \{ {\LL}x \, | \, x \in \mathbb{Z}^n \}.$$
The \emph{fundamental parallelepiped} of $\mathcal{L}$ with respect to $\LL$ is:
$$ \mathcal{P}(\mathcal{L}(\LL)) = \{ {\LL}x \, | \, x \in \mathbb{R}^n, 0 \leq x_i < 1
\}.$$ It is well known that the integer points of the fundamental
parallelepiped of a full-rank lattice form a system of representatives for $\mathbb{Z}^n / {\textrm{im}} {\LL}$, see e.g.~\cite[Chapter 9]{sinai}. 
The next proposition shows that these points can be seen as stable configurations in a chip-firing model.

\begin{proposition}
  Under the pairing $(\LL, \I)$,  the critical and superstable configurations of $S^+$ coincide and are precisely the integer points of the fundamental parallelepiped $\mathcal{P}(\mathcal{L}(\LL))$.
    \end{proposition}
\begin{proof}
  Let $f \in S^+$ be a superstable configuration and let $f = {\LL}x$.  Then $x$ is a superstable configuration in $R^+$.  Hence $x \geq 0$.  Suppose $x_i \geq 1$ for some $i$.  Then, $x - {\MM}e_i = x - e_i \geq 0$, contradicting the superstability of $x$.

  Therefore every superstable configuration $f$ is contained in $\mathcal{P}(\mathcal{L}(\LL))$.  Furthermore, the number of superstable configurations is the same as the number of integer points in $\mathcal{P}(\mathcal{L}(\LL))$ since both form a system of representatives for $\mathbb{Z}^n / {\textrm{im}} {\LL}$.  Hence the two collections must be the same.

  The argument above actually shows that any stable point of $S^+$ must be contained in $\mathcal{P}(\mathcal{L}(\LL))$.  As critical configurations are also stable and a system of representatives for $\mathbb{Z}^n / {\textrm{im}} {\LL}$, they indeed must coincide with the superstables.
  \end{proof}

\begin{example}\label{ex:I}
  Let $\LL$ be as in Example~\ref{ex:running} and let $\MM = \I$.  Then the superstable and critical configurations are:
 $$\begin{pmatrix*}[r] 1 \\ 0 \\1 \end{pmatrix*}, \begin{pmatrix*}[r] 0 \\ 1 \\ 0 \end{pmatrix*}, \begin{pmatrix*}[r] 0 \\ 0 \\ 0 \end{pmatrix*}, \begin{pmatrix*}[r] 2 \\ -1 \\ 2 \end{pmatrix*}.$$   
  \end{example}

\subsection{Combinatorial Laplacians}
The classical framework for chip-firing is on a graphical network with
distribution given by the graph Laplacian as discussed in Section~\ref{graphs}.  In~\cite{DKM}, a model for
chip-firing in higher dimensions was introduced.  The
underlying configuration of the system is a simplicial complex and the
commodities (chips, flow) are associated to ridges (codimension-one
faces) of the complex.   The distribution dynamics are given by a reduced
combinatorial Laplacian relating codimension-one faces to maximal
faces.  The algebraic aspects of chip-firing, i.e. the form and size
of the critical group were worked out in~\cite{DKM}, including the
connection to higher dimensional spanning trees of the complex.  It
remained an open problem to give a notion of critical or superstable configurations
in this context, namely sets of representative for the critical group
that reflect long-term stability dynamics of the system.  

The paradigm introduced here provides well motivated collections of
critical and superstable configurations.  We illustrate with an example and refer the reader to ~\cite{DKM} for a more thorough discussion of chip-firing in higher dimensions.

\begin{example} \label{ex:Tetra}
  Let $\Delta$ be the two-dimensional boundary complex of the
  tetrahedron with maximal faces: $\{123, 124, 134, 234\}$.  Faces of
  the complex will be oriented, we will use the standard orientation
  induced by the vertices.  In our example, this directs the edges
  from smaller vertex to larger vertex.  The two-dimensional faces are
  also oriented from smaller to larger vertices.  Envision traversing
  the triangular face $124$ by starting at vertex $1$ moving to vertex
  $2$ then vertex $4$ and returning back to vertex $1$.  Traversing
  the face in this way, we consider the edges $12$ and $24$ oriented
  in the same direction with respect to $124$ and oriented oppositely
  to the edge $14$ with respect to $124$.  The consistency of
  orientation of two edges contained in a common triangle will be
  reflected by a $+1$ or $-1$ in the combinatorial Laplacian.   

\begin{figure}[h]
\includegraphics[height=1.2in]{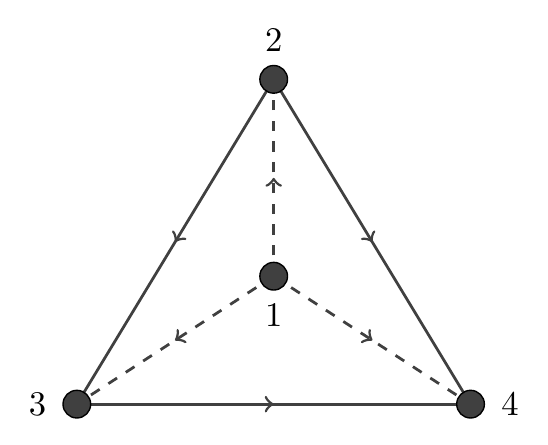}
\caption{The complex $\Delta$ is the boundary of the tetrahedron.  The three dashed edges represent the sink of the system.}  
\end{figure}

  As shown in~\cite{DKM}, the sink in this context consists of an
  entire spanning tree of the one-dimensional skeleton of $\Delta$.
  Choose the sink to be the one-dimensional tree $T$ with edges
  $\{12,13,14\}$.  With this choice, there will be three states in the
  system corresponding to the three edges not in the tree: $\{23, 24,
  34\}$.  The reduced combinatorial Laplacian is formed by deleting
  the rows and columns corresponding to the edges in the spanning
  tree.  In this case, the reduced combinatorial Laplacian $\LL$ of
  $\Delta$ with respect to $T$ is the matrix from
  Example~\ref{ex:running}, our running example:
\[
\LL = 
 \kbordermatrix{ & 23 & 24 & 34 \cr 23 & 2 & -1 & 1  \cr 24 & -1
  & 2 & -1 \cr 34 & 1 & -1 & 2}. 
\]

Configurations in $S^+$ are elements of $\mathbb{Z}^3$ which associate
an integer value to edges.  In this context, it is natural to consider
these values as flows along the edges (as opposed to numbers of
chips).  Firing an edge consists of subtracting the appropriate row
of the combinatorial Laplacian.  This can be interpreted as diverting
flow from an edge across triangles that contain it.  

Computing critical and superstable configurations requires a choice of
$\MM$.  Suppose $\MM = \I$.  Then we find the configurations of Example~\ref{ex:I}:
 $$\begin{pmatrix*}[r] 1 \\ 0 \\1 \end{pmatrix*}, \begin{pmatrix*}[r] 0 \\ 1 \\ 0 \end{pmatrix*}, \begin{pmatrix*}[r] 0 \\ 0 \\ 0 \end{pmatrix*}, \begin{pmatrix*}[r] 2 \\ -1 \\ 2 \end{pmatrix*}.$$   
Figure~\ref{fig:2Dcriticals} illustrates the three non-zero critical configurations.  The third image interprets the negative value in the critical configuration not as a negative flow but as flow in the opposite direction of the orientation of the edge.

\begin{figure}[h]
\includegraphics[height=1.4in]{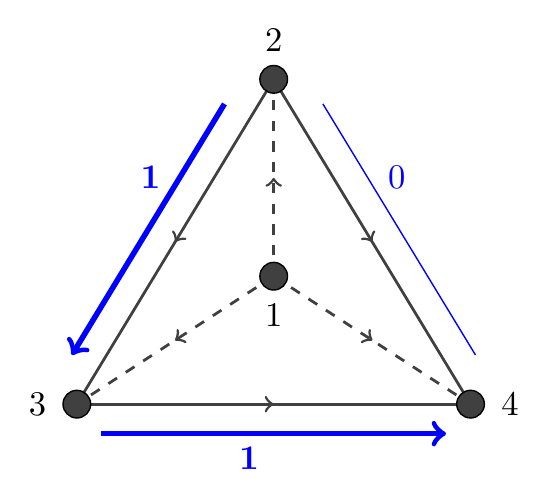}
\includegraphics[height=1.4in]{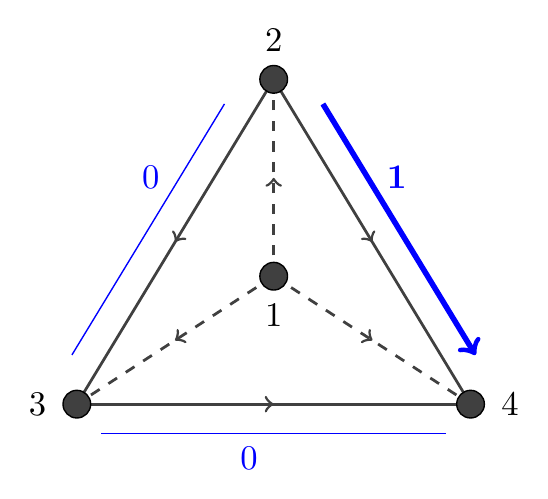}
\includegraphics[height=1.4in]{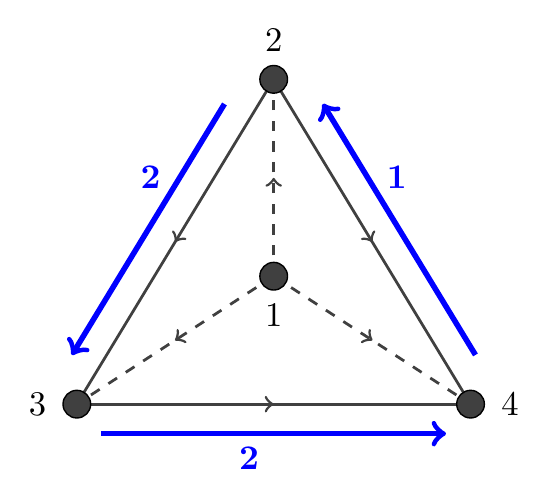}
\label{fig:2Dcriticals}
\caption{The three non-zero critical and superstable configurations of the tetrahedron with sink equal to all edges containing the vertex $1$ and under the pairing $(\LL, \I)$.}
\end{figure}

One narrative for this system is as a model of traffic flow.  The edges represent streets and the vertices are intersections.  Sink edges represent major boulevards that can handle large amounts of traffic.   As local streets become congested, traffic diverts to neighboring streets.  Informally, the system stabilizes when as much traffic as possible has been diverted to the major boulevards.  The third image of Figure~\ref{fig:2Dcriticals} is a solution in which the directionality of street $24$ has been reversed.  
\end{example}

We end with a general question given the new framework for chip-firing.  
\begin{question} Given $\LL$, are there natural choices for $\MM$?
\end{question}
We have seen that if $\LL$ is an $\MM$-matrix, then $\MM = \LL$ is a
natural choice that yields the classical model with non-negative valid
configurations.  We have also seen that we can always set $\MM = \I$
and achieve the integer points of the fundamental parallelepiped.
What are natural choices if $\LL$ is the reduced combinatorial
Laplacian of a simplicial complex, as in Example~\ref{ex:Tetra}?  The
matrix $\MM$ from Example~\ref{ex:running} is the reduced graph
Laplacian of the facet-to-facet dual graph of $\Delta$.  But, in
general, this dual graph will not have the correct size to match
$\LL$. \\

{\bf Acknowledgements}  The authors thank Harjasleen Malvai for helpful discussions on this project and acknowledge the Brown University summer UTRA program for her funding.  

%% ------------------------------------------------------------------

\end{document}